\documentclass[10pt]{amsart}
\usepackage{latexsym, amsmath,amssymb}

	\usepackage{url}

\renewcommand{\epsilon}{\varepsilon}
\newcommand{\newsection}[1]
{\subsection{#1}\setcounter{theorem}{0} \setcounter{equation}{0}
\par\noindent}

\newtheorem{theorem}{Theorem}

\newtheorem{lemma}[theorem]{Lemma}
\newtheorem{corr}[theorem]{Corollary}

\newtheorem{proposition}[theorem]{Proposition}
\newtheorem{deff}[theorem]{Definition}

\newcommand{\bth}{\begin{theorem}}
\newcommand{\ble}{\begin{lemma}}
\newcommand{\bcor}{\begin{corr}}

\newcommand{\bdeff}{\begin{deff}}

\newcommand{\bprop}{\begin{proposition}}
\newcommand{\ele}{\end{lemma}}
\newcommand{\ecor}{\end{corr}}
\newcommand{\edeff}{\end{deff}}

\newcommand{\eprop}{\end{proposition}}

\newcommand{\la}{\lambda}

\newcommand{\e}{\varepsilon}

\newcommand{\supp}{\text{supp }}
\renewcommand{\Pi}{\varPi}

\renewcommand{\epsilon}{\varepsilon}

\newcommand{\Rt}{{\mathbb R}^3}

\newcommand{\tidle}{\tilde}

\newcommand{\R}{{\mathbb R}}
\newcommand{\sqd}{\sqrt{-\Delta_g}}
\newcommand{\hint}{\int_{-\frac12}^{\frac12}}
\newcommand{\rnorm}{\|_{L^{\frac43}([-\frac12,\frac12])}}
\newcommand{\lnorm}{\|_{L^{4}([-\frac12,\frac12])}}
\newcommand{\Rtwo}{{\mathbb R}^2}
\newcommand{\sqdt}{\sqrt{-\Delta_{\tilde g}}}
\newcommand{\Stab}{\text{Stab}(\tilde \gamma)}
\newcommand{\Sstab}{S^{\text{Stab}}_\la}
\newcommand{\Sosc}{S^{\text{Osc}}_\la}
\newcommand{\Rth}{{\mathbb R}^3}

\begin{document}


\title[Endpoint geodesic restriction estimates]
{A few endpoint geodesic restriction estimates for eigenfunctions}
\thanks{The authors were supported in part by the NSF grant DMS-1069175. }
\author{Xuehua Chen}
\author{Christopher D. Sogge}
\address{Department of Mathematics,  Johns Hopkins University,
Baltimore, MD 21218}

\begin{abstract}
We prove a couple of new endpoint geodesic restriction estimates for eigenfunctions.  In the case of general $3$-dimensional compact
manifolds, after a $TT^*$ argument, simply by using the $L^2$-boundedness of the Hilbert transform on $\R$, we are able to improve the corresponding $L^2$-restriction bounds of 
Burq, G\'erard and Tzvetkov~\cite{burq} and Hu~\cite{Hu}.  Also, in the case of $2$-dimensional compact manifolds with nonpositive curvature, 
we obtain improved $L^4$-estimates
for restrictions to geodesics, which, by H\"older's inequality and interpolation, implies improved $L^p$-bounds for all exponents $p\ge 2$.
We do this by using
oscillatory integral theorems of H\"ormander~\cite{Hor}, Greenleaf and Seeger~\cite{GS} and Phong and Stein~\cite{PS0},  along with a simple geometric lemma (Lemma~\ref{lemma3.2}) about properties of the mixed-Hessian of the Riemannian distance function
restricted to pairs of geodesics in Riemannian surfaces. 
We are also able to get further improvements beyond our new results in three dimensions under the assumption of constant nonpositive curvature by exploiting the fact that in this case there are many totally geodesic submanifolds.
\end{abstract}

\maketitle

\newsection{Introduction}

The purpose of this paper is to prove a few new endpoint geodesic restriction estimates for eigenfunctions.  In the case of general $3$-dimensional compact
manifolds (no curvature assumptions), simply by using the $L^2$-boundedness of the Hilbert transform on $\R$, we are able to improve the corresponding $L^2$-restriction bounds of 
Burq, G\'erard and Tzvetkov~\cite{burq} and Hu~\cite{Hu}.  Also, in the case of $2$-dimensional compact manifolds with nonpositive curvature, by using
oscillatory integral theorems of H\"ormander~\cite{Hor}, Greenleaf and Seeger~\cite{GS} and Phong and Stein~\cite{PS0}, \cite{PS}  along with a simple geometric argument, we are able to show that, in this setting, we obtain improved $L^4$-estimates
for restrictions to geodesics.  The case of improved $L^2$-estimates, and hence, by interpolation, improved $L^p$-estimates for $2\le p<4$, in  this setting was obtained earlier
by the second author and Zelditch~\cite{SZ3}.  The improved $L^4$-bounds imply these by H\"older's inequality.  Also, the first author~\cite{Chen} earlier obtained a wide range of $\log$-improvements for restrictions to 
submanifolds $\Sigma\subset M$ under the assumption of nonpositive curvature, including such improvements for all exponents $p>2$ when the codimension of $\Sigma$ is equal to two.  In addition to these results, we shall also be able to obtain further endpoint improvements
for geodesic restriction estimates in 3-dimensions under the assumption of constant nonpositive curvature.

To state our results, if $(M,g)$ is a compact Riemannian manifold of dimension $n\ge2$, we let $\Pi$ denote the space of all unit-length geodesics $\gamma$.  Also,
if $\Delta_g$ is the Laplace-Beltrami operator associated with the metric $g$, then we shall consider eigenfunctions of frequency $\la$, i.e.,
$$-\Delta_ge_\la(x)=\la^2 e_\la(x),$$
and we are interested in whether we can improve certain estimates of the form
\begin{equation}\label{1.1}
\sup_{\gamma\in \Pi}\Bigl(\, \int_\gamma |e_\la|^p\, ds \, \Bigr)^{\frac1p}\le C_{\la}(1+\la)^{\sigma(n,p)}\|e_\la\|_{L^2(M)},
\end{equation}
with $ds$ denoting arc-length measure on $\gamma$.

Burq, G\'erard and Tzvetkov~\cite{burq} and Hu~\cite{Hu} (see also Reznikov~\cite{rez} for earlier work) showed that one can take $C_\la$ to be
independent of $\la>1$ and
\begin{equation}\label{1.2}
\sigma(2,p)=\begin{cases}
\frac14, \quad \text{if } \, \, 2\le p\le 4,
\\
\frac12-\frac1p, \quad \text{if } \, \, p\ge 4,
\end{cases}
\end{equation}
and
\begin{equation}\label{1.3}
\sigma(n,p)=
\begin{cases}
\frac{n-1}2-\frac1p, \quad \text{if } \, \, p\ge 2 \, \, \text{and } \, \, n\ge 4
\\
1-\frac1p, \quad \text{if } \, \, p>2 \, \, \, \text{and } \, \, n=3.
\end{cases}
\end{equation}
Burq, G\'erard and Tzvetkov~\cite{burq} also showed that these estimates are saturated by the highest weight spherical harmonics when $n\ge3$ on round spheres  $S^n$, 
as well as in the case of $2\le p\le 4$ when $n=2$, while in this case the zonal functions saturate the bounds for $p\ge4$.

In the case of $p=2$ and $n=3$, Burq, G\'erard and Tzvetkov~\cite{burq} and Hu~\cite{Hu} also obtained the following estimates
$$\Bigl(\int_\gamma |e_\la|^2 \, ds\Bigr)^{\frac12}\le C (\log(2+\la))^{\frac12}(1+\la)^{\frac12}\|e_\la\|_{L^2(M)},$$
which are not saturated on the sphere, $S^3$.
Our first result is that one can dispose of the $\log$-factor and obtain the following result which is sharp for the aforementioned reason.

\begin{theorem}\label{theorem1}  Fix a $3$-dimensional compact Riemannian manifold $(M,g)$.  Then there is a uniform
constant $C=C(M,g)$ such that
\begin{equation}\label{1.4}
\sup_{\gamma\in \Pi}\Bigl(\int_\gamma |e_\la|^2\, ds\bigr)^{\frac12}\le C(1+\la)^{\frac12}\|e_\la\|_{L^2(M)}.
\end{equation}
\end{theorem}

We remark that, in higher dimensions $n\ge4$, the problem of what are the optimal $L^2(\Sigma)$ restriction estimates when $\Sigma$ is a submanifold of $M$ of codimension
two is open.  We are able to prove new results in dimension $n=3$ in this case  due to the fact that after using the Hadamard parametrix and a simple $TT^*$ argument
the needed punchline just follows from the $L^2(\R)$-boundedness of the Hilbert transform.

The second author and Zelditch showed in \cite{SZ3}, motivated in part by earlier work of Bourgain~\cite{bourgainef} and the second author~\cite{Sokakeya}, that if
$n=2$ and one assumes that $(M,g)$ has nonpositive curvature then one can improve \eqref{1.1} if $p=2$, and, hence, by interpolation if $2\le p<4$, to get 
$L^2(M)\to L^p(\gamma)$ bounds which are $o(\la^{\frac14})$.  Our second result is
the following slightly stronger one saying that one  can also obtain these improvement when $p=4$:

\begin{theorem}\label{theorem2}  Fix a $2$-dimensional compact manifold of nonpositive curvature.  Then if $\{e_{\la_j}\}$ is an orthonormal basis of eigenfunctions
of frequency $\la_j$, we have
\begin{equation}\label{1.5}
\limsup_{\la_j\to \infty}\left( \sup_{\gamma\in \Pi}\la_j^{-\frac14}\Bigl(\int_\gamma |e_{\la_j}|^4 \, ds\Bigr)^{\frac14}\right)=0.
\end{equation}
\end{theorem}

We remark that, to the best of our knowledge, this is the first result establishing an improvement under general (e.g., non-arithmetic) assumptions of an estimate that is saturated
both by zonal functions (ones concentrating at points) and highest weight spherical harmonics (ones concentrating along periodic orbits).  Almost all the results involve improving
estimates such as the ones in \cite{S2} that are saturated by the zonal functions and therefore involve relatively large exponents.  See, e.g., 
\cite{BSSY}, \cite{Chen}, \cite{HT}, 
 \cite{STZ} and \cite{SZDuke}.  The only
general result about improving estimates that are saturated by functions concentrating on periodic orbits seem to be the ones in \cite{SZ3} (see also \cite{SZ2}
for related work), which also are just for the $2$-dimensional case.  On the torus, though, there has been recent work showing that much stronger results hold.  See, e.g., 
\cite{Berard},
\cite{B1} and \cite{BRud}.  Toth and Zelditch~\cite{TZ} also showed that there are substantial improvements under certain ergodic assumptions.  We also would like to point out that
the problem of finding improvements for estimates that are saturated by the highest weight spherical harmonics is related to the problem of obtaining lower bounds
for the size of nodal sets of eigenfunctions (see \cite{CM}, \cite{SZ} and \cite{HeS}).

In the main, in establishing \eqref{1.5}, we follow the template provided the proof of the earlier corresponding weaker estimates in \cite{SZ3} where the $L^4(\gamma)$-norms are replaced by $L^2(\gamma)$ ones,
except instead of using microlocal analysis to essentially reduce matters to an estimate involving the stabilizer subgroup of deck transformations associated with
a given $\gamma\in \Pi$, we now make this reduction by applying oscillatory integral theorems of H\"ormander's~\cite{Hor}, Greenleaf and Seeger~\cite{GS}
and Phong and Stein~\cite{PS0}   along a simple
fact (Lemma~\ref{lemma3.2}) about the Hessian of the distance function restricted to geodesics when $n=2$.  
%We are unable to prove for $n=3$ the corresponding improvements of \eqref{1.4} under
%the assumption of nonpositive sectional curvatures due to the fact that corresponding fact about the Hessian of the Riemannian distance function restricted to pairs
%of geodesics seems to be much more complicated.  However, if a suitable variant were valid under these assumptions our arguments would then show that the bounds in \eqref{1.4} could
%be improved to be $o(\la)$ as $\la\to \infty$.

It is more difficult to prove the facts about the Hessian of the distance function in 3-dimensions that our techniques require.  However, if we assume
constant curvature, we can obtain the following endpoint estimates which are  the analog of Theorem~\ref{theorem2} in this case:

\begin{theorem}\label{theorem3}  Fix a compact $3$-dimensional compact manifold of {\em constant} nonpositive curvature.  
Then if $\{e_{\la_j}\}$ is an orthonormal basis of eigenfunctions
of frequency $\la_j$, we have
\begin{equation}\label{1.6}
\limsup_{\la_j\to \infty}\left( \sup_{\gamma\in \Pi}\la_j^{-\frac12}\Bigl(\int_\gamma |e_{\la_j}|^2 \, ds\Bigr)^{\frac12}\right)=0.
\end{equation}
\end{theorem}

Like in the case of Theorem~\ref{theorem2} above, this result is an improvement, under the current hypotheses, of an estimate that is saturated by both the zonal functions and the highest weight spherical harmonics on $S^3$ in this case.

\bigskip

%These improved $L^p$-restriction theorems are also related to much work on obtaining improved $L^p(M)$ estimates compared to the general ones in \cite{S2} under
%certain geometric assumptions.  There has been much work recently on these types of problems and on applications.  See, for instance \cite{bourgainef}, 
%\cite{BRud},
%\cite{CM}
%\cite{Sokakeya}, \cite{SZ}, \cite{SZ3}, \cite{STZ}.  See also a recent paper of Toth and Zelditch~\cite{TZ} for related work on much improved restriction estimates
%under certain ergodic assumptions.

The authors would like to thank their colleague J.-E. Chang and their former colleagues W. P. Minicozzi II and S. Zelditch for helpful conversations.

\newsection{Proof of an improved endpoint estimate in $3$-dimensions}

In this section we shall prove Theorem~\ref{theorem1}.  Fix an even function $\rho\in {\mathcal S}(\R)$ satisfying
\begin{equation}\label{2.1}
\rho(0)=1 \quad \text{and } \, \Hat \rho(\tau)=0, \, \, \, \text{if } \, |\tau|\ge \delta,
\end{equation}
where $\delta>0$ is smaller than a quarter of the injectivity radius of our $3$-dimensional compact Riemannian manifold $(M,g)$.  It then follows that
$$\rho(\la-\sqd)e_\la=e_\la,$$
and consequently, we would have \eqref{1.4} if we could show that there is a uniform constant $C=C(M,g)$ such that
\begin{equation}\label{2.2}
\Bigl(\int_\gamma |\rho(\la-\sqd)f|^2 \, ds\Bigr)^{\frac12}\le C(1+\la)^{\frac12}\|f\|_{L^2(M)}
\end{equation}
whenever $\gamma$ is a segment of a geodesic of length $\ell$ equal to  half the injectivity radius of $(M,g)$.

Let us choose geodesic normal coordinates about the center of $\gamma$ so that
$$\gamma(s)=(s,0,0), \quad |s|\le \ell/2.$$
In these coordinates the Riemannian distance, $d_g(\gamma(t),\gamma(s))$ between two points on the segment is just
\begin{equation}\label{2.3}
d_g(\gamma(t),\gamma(s))=|t-s|.
\end{equation}

Let $\chi\in {\mathcal S}(\R)$ denote the square of $\rho$ and $\chi(\la-\sqd)$ the associated multiplier operator.  Then if $\chi(\la-\sqd)(x,y)$, $x,y\in M$, denotes
its kernel, it follows that \eqref{2.2} would be a consequence of the following estimate
\begin{equation}\label{2.4}
\left(\int_{-\ell/2}^{\ell/2}\Bigl|\int_{-\ell/2}^{\ell/2}\chi(\la-\sqd)(\gamma(t),\gamma(s)) \, h(s)\, ds\Bigr|^2 \, dt\right)^{\frac12}
\le C(1+\la)\|h\|_{L^2([-\ell/2,\ell/2])}.
\end{equation}

To prove this inequality we first write
\begin{align}\label{2.5}
\chi(\la-&\sqd)(x,y)
\\ &=\frac1{2\pi}\int_{-\infty}^\infty \Hat \chi(\tau)e^{i\tau \la}\bigl(e^{-i\tau\sqd}\bigr)(x,y)\, d\tau \notag
\\
&=\frac1\pi \int_{-\infty}^\infty \Hat \chi(\tau)e^{i\la\tau}\bigl(\cos \tau\sqd\bigr)(x,y)\, d\tau - \chi(\la+\sqd)(x,y),  \notag
\end{align}
where $\bigl(e^{-i\tau\sqd}\bigr)(x,y)$ and $\bigl(\cos \tau \sqd\bigr)(x,y)$ denote the kernels of the operators $e^{-i\tau\sqd}$ and
$\cos \tau\sqd$, respectively.  Since 
\begin{equation}\label{2.6}
|\chi(\la + \sqd)(x,y)|\le C_N(1+\la)^{-N}, \quad N=1,2,3,\dots,
\end{equation}
in order to prove \eqref{2.4} it is enough to show that if we replace the kernel involved by the restriction to $\gamma\times \gamma$ of the second to last term in \eqref{2.5} then we have the analog of
\eqref{2.4}.  In other words if we set
\begin{equation}\label{2.7}
K(t,s)=\frac1\pi \int_{-\infty}^\infty \Hat \chi(\tau)e^{i\la \tau} \bigl(\cos \tau \sqd\bigr)(\gamma(t),\gamma(s)) \, d\tau, \quad |s|, |t|\le \ell/2,
\end{equation}
then we have reduced matters to showing that
\begin{equation}\label{2.8}
\left(\int_{-\ell/2}^{\ell/2}\Bigl|\int_{-\ell/2}^{\ell/2}K(t,s)\, h(s)\, ds \Bigr|^2 \, dt\right)^{\frac12}\le C(1+\la)\|h\|_{L^2([-\ell/2,\ell/2])}.
\end{equation}

Since $\Hat \chi(\tau)=(2\pi)^{-1}\bigl(\Hat \rho * \Hat \rho\bigr)(\tau)$ vanishes when $|\tau|$ is larger than half the injectivity radius of $(M,g)$, we may
use the Hadamard parametrix (see \cite{Hor3} and \cite{SoggeHang}) to evaluate our kernel in \eqref{2.7}.  
In particular,  because of \eqref{2.3}, we have that for $\tau \in \text{supp } \Hat \chi$ and $|s|, |t|\le \ell/2$ we can write
\begin{multline}\label{2.9}
\bigl(\cos \tau\sqd\bigr)(\gamma(t),\gamma(s))=(2\pi)^{-3}w(t,s) \int_{\Rt}e^{i(t-s)\xi_1} \cos (\tau|\xi|) \, d\xi
\\
+\sum_{\pm}\int_{\Rt}e^{i(t-s)\xi_1}e^{\pm i\tau|\xi|} a_\pm (\tau, t,s; |\xi|)\, d\xi + R(s,t),
\end{multline}
where the leading coefficient $w(t,s)$ satisfies
\begin{equation}\label{2.10}
w(t,s)\in C^\infty([-\ell/2,\ell/2]\times [-\ell/2,\ell/2]) \quad \text{and } \, \, w(t,t)\equiv 1, \, \, |t|\le \ell/2,
\end{equation}
and where the remainder term may be taken to be bounded, i.e.,
\begin{equation}\label{2.11}
|R(t,s)|\le C, \quad |t|, |s|\le \ell/2,
\end{equation}
and, finally, where $a_\pm$ are symbols of order $-2$, and so in particular satsify
$$|\partial^j_\tau a_\pm(\tau,t,s; |\xi|)| \le C_j(1+|\xi|)^{-2}, \quad \text{if } \, \tau \in \supp \chi, \, \, \text{and } \, j=0,1,2,\dots,$$
(see \cite[Theorem 3.1.5]{SoggeHang}).
Based on this we deduce that for each $N=1,2,3,\dots$ there is a constant $C_N$ such that
\begin{multline*}\Bigl|\int_{\Rt}\int_{-\infty}^\infty \Hat \chi(\tau) e^{i(t-s)\xi_1}a_{\pm}(\tau,t,s; |\xi|) e^{i\la\tau}e^{\pm i\tau|\xi|}\, d\tau d\xi\Bigr|
\\
\le C_N\int_{\Rt}(1+|\la-|\xi|\, |)^{-N}(1+|\xi|)^{-2}\, d\xi.
\end{multline*}
If $N>1$ the last expression is uniformly bounded independently of $\la$, and, consequently, by \eqref{2.7}, \eqref{2.9} and \eqref{2.11}, we have
\begin{multline}\label{2.12}
K(t,s)= (2\pi)^{-3}\frac{w(t,s)}\pi \int_{-\infty}^\infty \int_{\Rt}\Hat \chi(\tau) e^{i(t-s)\xi_1} e^{i\la \tau}\cos (\tau |\xi|)\, d\xi d\tau
+O(1), 
\\ |t|, |s|\le \ell/2,
\end{multline}
where the $O(1)$ error term is uniformly bounded for all $\la>0$.

If we replace $\cos(\tau|\xi|)$ by $e^{i\tau|\xi|}$ in the above integral then it is easy to see that the resulting expression is $O((1+\la)^{-N})$ for each $N=1,2,3,\dots$
since $\chi\in {\mathcal S}(\R)$.   Therefore, by Euler's formula if we let
$$K_0(t,s)=(2\pi)^{-4}w(t,s)\int_{-\infty}^\infty\int_{\Rt}\Hat \chi(\tau)e^{i(t-s)\xi_1}e^{i\tau(\la-|\xi|)} \, d\xi d\tau,$$
then we would have \eqref{2.8} and therefore complete the proof of Theorem~\ref{theorem1} if we could 
show that
\begin{equation}\label{2.13}
\left(\int_{-\ell/2}^{\ell/2}\Big|\int_{-\ell/2}^{\ell/2}K_0(t,s)\, h(s)\, ds\Bigr|^2 \, dt\right)^{\frac12}\le C(1+\la)\|h\|_{L^2([-\ell/2,\ell/2])}.
\end{equation}
We can evaluate $K_0$ if we recall the following formula for the Fourier transform of Euclidean surface measure, $d\omega$, on the unit sphere $S^2\subset \Rt$
\begin{equation}\label{2.14}\int_{S^2}e^{i\omega\cdot \eta}d\omega =4\pi \frac{\sin |\eta|}{|\eta|}.
\end{equation}
Using this, polar coordinates and Fourier's inversion formula for $\R$ we deduce that for $|t|,|s|\le \ell/2$,
\begin{align}\label{2.15}
K_0(t,s)&=(2\pi)^{-3}w(t,s)\int_{\Rt}\chi(\la-|\xi|)e^{i(t-s)\xi_1}\, d\xi
\\
&=\frac{w(t,s)}{2\pi^2}\int_0^\infty \chi(\la-r)\frac{\sin(t-s)r}{t-s} \, r dr \notag
\\
&=
\frac1{2\pi^2}\int_0^\infty \chi(\la-r)\frac{\sin (t-s)r}{t-s}\, rdr + O((1+\la)),
\notag
\end{align}
using \eqref{2.10} and the fact that for $\la>0$ we have
\begin{equation}\label{2.16}
\int_0^\infty |\chi(\la-r)|\, rdr \le C(1+\la).
\end{equation}

Using \eqref{2.15} and Minkowski's integral inequality, we conclude that we would have \eqref{2.13} if 
\begin{equation*}
\int_0^\infty |\chi(\la-r)| r \, \Bigl(\int_{-\ell/2}^{\ell/2}\Bigl|\int_{-\ell/2}^{\ell/2}\frac{\sin(t-s)r}{t-s} \, h(s)\, ds \Bigr|^2 \, dt\Bigr)^{\frac12} \, dr
%\\
\le C(1+\la)\|h\|_{L^2([-\ell/2,\ell/2])}.
\end{equation*}
This just follows from \eqref{2.16}  and the well known uniform bounds
\begin{equation}\label{2.17}
\int_{-\infty}^\infty \Bigl|\int_{-\infty}^\infty \frac{\sin(t-s)r}{t-s}\, f(s)\, ds\Bigr|^2 \, dt \le C\|f\|_{L^2(\R)}^2, \quad r>0,
\end{equation}
since $\sin (t-s)r/(t-s)$ is (essentially) the  Dirichlet kernel for the partial summation of one-dimensional Fourier series
(see e.g., \cite{StS} p. 37).  Alternatively, \eqref{2.17} also follows from Euler's formula and the $L^2$-boundedness of
the Hilbert transform
$$\int_{-\infty}^\infty \Bigl| \, \text{P.V.}\int_{-\infty}^\infty \frac{f(s)}{t-s}\, ds \, \Bigr|^2 \, dt \le C\|f\|_{L^2(\R)}^2.$$

This completes the proof of Theorem~\ref{theorem1}.

\newsection{Proof of an improved endpoint estimate for $2$-dimensional manifolds with nonpositive curvature}

Let us now turn to the proof of Theorem~\ref{theorem2}.  We are now assuming that $(M,g)$ is a compact two-dimensional Riemannian manifold with nonpositive curvature.  To prove Theorem~\ref{theorem2} we shall show that given $\e>0$ we can find a number $\Lambda(\e)<\infty$ so that
\begin{equation}\label{3.1}
\Bigl(\int_\gamma |e_\la|^4\, ds\Bigr)^{\frac14}\le \e \la^{\frac14}\|e_\la\|_{L^2(M)}, \quad \la\ge \Lambda(\e), \, \, \gamma\in \Pi,
\end{equation}
if, as before, $e_\la$ is an eigenfunction of $-\Delta_g$ of frequency $\la$ and $\Pi$ is the space of all unit-length geodesics in $(M,g)$.  To simplify some of the calculations, after perhaps multiplying the metric by a constant, we shall, as we may, assume from now on that the injectivity radius of $(M,g)$ is ten or more.

As in the proof of Theorem~\ref{theorem1}, we shall prove \eqref{3.1} by estimating an operator which reproduces eigenfunctions.  So if $\rho\in {\mathcal S}(\R)$, as before, is an even
function satisfying \eqref{2.1} with now $\delta=1/2$, we now note that
$$\rho(T(\la-\sqd))e_\la=e_\la.$$
In the proof of Theorem~\ref{theorem1} we just took $T$ to be equal to one, while now our choice of $T$ depends on $\e$.

Specifically, to prove \eqref{3.1} we note that it suffices to show that there is a uniform constant $C$, independent of $T\gg 1$, and also constants $C_T$ depending on such $T$ so that
\begin{multline}\label{3.2}
\Bigl(\int_\gamma |\rho(T(\la-\sqd))f|^4\, ds\Bigr)^{\frac14}\le CT^{-\frac14}(1+\la)^{\frac14}\|f\|_{L^2(M)}
\\
+C_T(1+\la)^{\frac3{16}}\|f\|_{L^2(M)}, \quad \gamma \in \Pi.
\end{multline}
If $T\gg 1$ is chosen so that $CT^{-\frac14}<\e$, we then obtain \eqref{3.1}.  We shall obtain bounds of the form \eqref{3.2} first for a given $\gamma\in \Pi$, and then, at the end, see how we can use the argument for this special case to obtain bounds which are uniform as $\gamma$ ranges over $\Pi$ by a simple compactness argument.  

As we shall see, the first term in the right side of \eqref{3.2} comes from a sum involving the stabilizer group of deck transformations for the lift of $\gamma$ to the universal cover which just involves a few terms (typically only one).  Due to the paucity of the number of terms in the sum that arises, we shall be able to estimate the contribution of the stabilizer group very easily just by using size estimates based on stationary phase.  In some ways this part of the proof is akin to an argument going back to B\'erard~\cite{Berard}, but for different reasons, B\'erard was able to treat all terms in his sum this way.  In \cite{SZ3} the second author and Zelditch, in proving $L^2(\gamma)$-restriction estimates, were able to treat the terms not coming from the stabilize group via an argument using microlocal analysis.   Estimating these terms is more delicate, since if the curvature of $(M,g)$ is strictly negative, the number of terms arising grows
exponentially in $T$.  The argument in \cite{SZ3} does not seem to be able to prove the $L^4$-endpoint estimate.  We shall get around this by using oscillatory integral theorems in
\cite{Hor}, \cite{GS} and \cite{PS0} along with a simple geometric lemma about the mixed Hessian of the distance function in two-dimensions.

Let us now get down to the details.  If $\chi(\tau)=(\rho(\tau))^2$, then, as in the proof of Theorem~\ref{theorem1}, we see via a simple $TT^*$ argument that \eqref{3.2} would follow from showing that if $\gamma=\gamma(t)$, $|t|\le 1/2$, is a parameterization by arclength of our fixed $\gamma\in \Pi$, then for $\la>1$ we have
\begin{multline}\label{3.3}
\Bigl(\hint \Bigl| \hint \chi\bigl(T(\la-\sqd)\bigr)(\gamma(t),\gamma(s))\, h(s) \, ds \Bigr|^4\, dt \Bigr)^{\frac14}
\le CT^{-\frac12}\la^{\frac12}\|h\rnorm
\\ 
+C_T\la^{\frac38}\|h\rnorm.
\end{multline}
Here, $\chi(T(\la-\sqd))(x,y)$, $x,y\in M$, denotes the kernel of the multiplier operator $\chi(T(\la-\sqd))$.

Note that since we are assuming \eqref{2.1} with $\delta=1/2$ it follows that 
$$\Hat \chi(\tau)=(2\pi)^{-1} (\Hat \rho * \Hat \rho)(\tau)=0, \quad \text{if } \, \, |\tau|\ge 1.$$  Pick a bump function $\beta\in C^\infty_0(\R)$ satisfying
$$\beta(\tau)=1 \, \, \, \text{for } \, \, |\tau|\le 3/2, \, \, \text{and } \, \, \beta(\tau)=0, \, \, \, |\tau|\ge 2.$$
We then may write
\begin{multline*}
\chi(T(\la-\sqd))(x,y)=\frac1{2\pi T}\int \beta(\tau)\Hat \chi(\tau/T) e^{i\la\tau} \bigl(e^{- i\tau\sqd}\bigr)(x,y)\, d\tau
\\
+\frac1{2\pi T}\int (1-\beta(\tau)) \, \Hat \chi(\tau/T) e^{i\la\tau} \bigl(e^{- i\tau\sqd}\bigr)(x,y)\, d\tau
=K_0(x,y)+K_1(x,y).
\end{multline*}
Since we are assuming that the injectivity radius is ten or more and $\beta(\tau)=0$ for $|\tau|\ge2$ it is not difficult to estimate $K_0$.  In particular
since the Riemannian distance between $\gamma(t)$ and $\gamma(s)$ equals $|t-s|$, the proof of Lemma 5.1.3 in \cite{S3} shows that
there is a uniform constant $C$ so that for $\la$, $T>1$, we have
$$|K_0(\gamma(t),\gamma(s))|\le CT^{-1}\la^{\frac12}|t-s|^{-\frac12}.$$
Hence, by the Hardy-Littlewood fractional integral theorem we have
$$\Bigl(\hint \Bigl| \hint \Bigl| K_0(\gamma(t),\gamma(s))\, h(s) \, ds\Bigr|^4\, dt \Bigr)^{\frac14}
\le CT^{-1}\la^{\frac12}\|h\rnorm,$$
which is better than the bounds posited in \eqref{3.3}.
As a result, if $W_\la$ is the integral operator associated with restricting $K_1$ to $\gamma\times\gamma$, i.e.,
\begin{equation*}%\label{3.4}
W_\la h(t)=\frac1{2\pi T}\int_{-\infty}^\infty \hint (1-\beta(\tau)) \, \Hat \chi(\tau/T) e^{i\la\tau} \bigl(e^{- i\tau\sqd}\bigr)(\gamma(t),\gamma(s))\, h(s) \, ds d\tau,
\end{equation*}
then it suffices to show that
\begin{equation*}%\label{3.5}
\| W_\la h\lnorm 
\le CT^{-\frac12}\la^{\frac12}\|h\rnorm
+C_T\la^{\frac38}\|h\rnorm. 
\end{equation*}
Since the kernel of $\chi(T(\la+\sqd))$ is $O(\la^{-N})$ with constants independent of  $T, \la>1$, by Euler's formula, we would have this inequality if we
could show that
\begin{equation}\label{3.4}
\| S_\la h\lnorm 
\le CT^{-\frac12}\la^{\frac12}\|h\rnorm
+C_T\la^{\frac38}\|h\rnorm,
\end{equation}
with
\begin{equation}\label{3.5}
S_\la h(t)=\frac1{\pi T}\int_{-\infty}^\infty \hint (1-\beta(\tau)) \, \Hat \chi(\tau/T) e^{i\la\tau} \bigl(\cos \tau\sqd\bigr)(\gamma(t),\gamma(s))\, h(s) \, ds d\tau.
\end{equation}

It is at this point where we need to use our hypothesis that $(M,g)$ is of everywhere nowhere nonpositive curvature.  By a theorem of Hadamard
(see \cite{Chavel}, \cite{do Carmo}), for each point $p\in M$, the exponential map at $p$, $\exp_p$, sending the tangent space at $p$, $T_pM$ to $M$ is a universal
covering map.  
%We may take $p$ to be a point $\gamma(s_0)$ on our geodesic which is arbitrarily close to the center and has the property that if $\gamma(s)=\gamma(s_0)$
%for some $|s|\le 2T$ then $d\gamma(s)/ds = d\gamma(s_0)/ds$.  We can do this since $s\to \gamma(s)$, $|s|\le 2T$ crosses the segment $\{\gamma(s): \, |s|\le 1/2\}$ 
%finitely many times and we can remove the nonperiodic crosses by moving the point a bit from the center if the center does not have only periodic crossings (or none).  
%Since we can take $s_0$ to be arbitrarily close to zero, we shall just assume, for simplicity, in what follows that $s_0=0$ which just amounts to changing our initial geodesic a bit.
%As we mentioned before, at the end, we shall argue that the estimates that we are proving can be taken to be uniform of small perturbations of a given element of $\Pi$.
For us, it is natural for us to take this distinguished point to be the center of our geodesic segment.
Thus, we shall take  $p=\gamma(0)$ which defines the covering map that we shall use,
\begin{equation}\label{3.6}
\kappa = \exp_{\gamma(0)}: \Rtwo \simeq T_{\gamma(0)}M\to M.
\end{equation}
%and we are assuming that if $\gamma(t)$ loops back to $\gamma(0)$ for some time $|t|\le 2T$ then it loops back smoothly.

If $\tilde g=\kappa^*g$ denotes the pullback of the metric $g$ on $M$ to $\Rtwo$ it follows that $\kappa$ is a local isometry and straight lines through the origin are geodesics
for the metric $\tilde g$.  We may assume that $\kappa(t(1,0))=\gamma(t)$, $t\in \R$, where $\gamma\in M$ is the geodesic containing the geodesic segment $\{\gamma(t): \, 
|t|\le 1/2\}$ of our fixed element of $\Pi$.  In other words, we are identifying our segment with the part of the $x_1$ axis, $\{(t,0): \, |t|\le 1/2\}$ which lies in the interior of the Dirichlet domain associated  with $(M,g)$ and our distinguished point.

Next, we consider the deck transformations which are the set of diffeomorphisms $\alpha: \, \Rtwo \to \Rtwo$ having the property that
$$\kappa\circ \alpha = \kappa.$$
The set of these maps form a group which we shall denote by $\Gamma$.  Since $\alpha^*\tilde g=\tilde g$, $\alpha\in \Gamma$, it follows that any deck transformation
preserves distances and angles.  Hence, if $t\to \tilde \gamma_1(t)$, $t\in \R$, is a geodesic parameterized by arclength then so is its image $t\to \alpha(\tilde \gamma_1(t))$,
$t\in \R$, when $\alpha$ is a deck transformation.  Also, of course $M\simeq \Rtwo / \Gamma$.

The reason that we have brought these things up is that there is an analog of the classical Poisson summation formula for the torus which relates wave kernels
for $(M,g)$ to ones for $(\Rtwo, \tilde g)$, which will allow us to compute the kernel of our operator in \eqref{3.5}.  In particular, if as above $\tilde \gamma(t)=\{(t,0): \, |t|\le 1/2\}$ denotes
the lift of our geodesic segment in $(M,g)$ to the universal cover $(\Rtwo, \tidle g)$, we have
\begin{equation}\label{3.7}
\bigl(\cos \tau\sqd\bigr)(\gamma(t),\gamma(s))=\sum_{\alpha\in \Gamma}\bigl(\cos \tau \sqdt \bigr)(\tilde \gamma(t),\alpha(\tilde \gamma(s))), \quad |s|, |t|\le 1/2.
\end{equation}
Here $(\cos \tau \sqdt )(x,y)$, $x,y\in \Rtwo$, denotes the kernel of the cosine transform associated with the metric $\tilde g$, so that if $dV_{\tilde g}=(\det \tilde g_{jk}(x))^{\frac12}dx$
denotes the associated volume element and we set for $f\in C^\infty(\Rtwo)$
$$u(\tau,x)=\int_{\Rtwo} \bigl(\cos \tau\sqdt \bigr)(x,y)\, f(y)\, dV_{\tilde g}(y),$$
then $u$ is the unique solution of the following Cauchy problem associated with the d'Alembertian for $\tilde g$,
$$(\partial^2_\tau-\Delta_{\tilde g})u=0, \quad u(0,x)=f(x), \quad \partial_\tau u(0,x)=0.$$
Therefore, by the Huygens principle we have 
\begin{equation}\label{3.8}
\bigl(\cos \tau\sqdt \bigr)(x,y)=0 \quad \text{if } \, \, d_{\tilde g}(x,y)>\tau,
\end{equation}
with $d_{\tilde g}$ denoting the Riemannian distance coming from the metric $\tilde g$.

Using this formula we can rewrite the operator in \eqref{3.5} that remains to be studied.  Specifically, for $|t|\le 1/2$, we have
$$S_\la h(t)=\frac1{\pi T}\sum_{\alpha\in \Gamma}\int_{-\infty}^\infty \hint (1-\beta(\tau)) \, \Hat \chi(\tau/T) e^{i\la\tau} 
\bigl(\cos \tau\sqdt\bigr)(\tilde \gamma(t), \alpha(\tilde \gamma(s)))\, h(s) \, ds d\tau.$$
We shall be able to use the Hadamard parametrix for $(\Rtwo, \tilde g)$ to compute each summand with the necessary precision.  The arguments from the last section suggest that when we do so if $\{\alpha(\tilde \gamma(s)): \, |s|\le 1/2\}$ is contained in the $x_1$-axis (i.e., the infinite geodesic
$\tilde \gamma = \{(t,0): \, t\in \R\}$ for $\tilde g$), then the resulting integral operator will have trivial oscillations.  The set of all such $\alpha$
is the stabilizer group for $\tilde \gamma$, $\Stab$, of all $\alpha \in \Gamma$ so that $\alpha(\tilde \gamma)=\tilde \gamma$.  This is a 
cyclic subgroup of $\Gamma$ which is just the identity if the base geodesic in $M$, $t\to \gamma(t)$, $t\in \R$ is not periodic (which is the
typical case).

As a result of these considerations we write
$$S_\la h(t)=\Sstab h(t) + \Sosc h(t),\quad |t|\le 1/2,$$
where for $|t|\le 1/2$
\begin{multline}\label{3.9}\Sstab h(t)= 
\\
\frac1{\pi T}\sum_{\alpha\in \Stab}\int_{-\infty}^\infty \hint (1-\beta(\tau)) \, \Hat \chi(\tau/T) e^{i\la\tau} 
\bigl(\cos \tau \sqdt\bigr)(\tilde \gamma(t), \alpha(\tilde \gamma(s)))\, h(s) \, ds d\tau,
\end{multline}
and
\begin{multline}\label{3.10}\Sosc h(t)= 
\\
\frac1{\pi T}\sum_{\alpha\in \Gamma \backslash \Stab}\int_{-\infty}^\infty \hint (1-\beta(\tau)) \, \Hat \chi(\tau/T) e^{i\la\tau} 
\bigl(\cos \tau\sqdt\bigr)(\tilde \gamma(t), \alpha(\tilde \gamma(s)))\, h(s) \, ds d\tau.
\end{multline}
For a generic geodesic segment $\gamma\in \Pi$, as we mentioned before, the sum in \eqref{3.9} will just consist of one term,  $\alpha=Identity$, which
will be trivial for $|t|\le 1/2$ since the integrand vanishes when $|\tau|\le 3/2$.  Even if $\gamma$ is part of a periodic geodesic in $M$, as we shall see
it follows from \eqref{3.8} that the number of nonzero summands is $O(T)$.  It will essentially be controlled by the first term in the right side of \eqref{3.4}.
The other sum, \eqref{3.10}, will have a number of nonzero
terms which grows exponentially in $T$, though, which could create problems.  Fortunately, each such term is an oscillatory integral and we
shall be able use a simple geometric lemma and oscillatory integral estimates from \cite{Hor}, \cite{GS} and \cite{PS0}
to prove that it is controlled by the second term in the right side of \eqref{3.4}, which has a constant that is badly behaved in $T$, but an
improved power of $\lambda$.

To prove these assertions we need the following simple lemma which follows from the Hadamard parametrix and stationary phase:

\begin{lemma}\label{lemma3.1}
Given $\alpha\in \Gamma$ set
\begin{multline}\label{3.11}
K^\gamma_{T,\lambda,\alpha}(t,s)
\\
=\frac1{\pi T}\int_{-\infty}^\infty (1-\beta(\tau)) \, \Hat \chi(\tau/T) e^{i\la\tau} 
\bigl(\cos \tau \sqdt\bigr)(\tilde \gamma(t), \alpha(\tilde \gamma(s)))\, d\tau, \, \, |s|, |t|\le 1/2.
\end{multline}
We then have for a constant $C_T$ independent of $\gamma$ and $\la, T>1$
\begin{equation}\label{3.12}
|K^\gamma_{T,\la,\alpha}(t,s)|\le C_{T,N}\la^{-N} \quad \text{if } \, \alpha=Identity, \, \, |s|, |t|\le 1/2, \quad N\in {\mathbb N}.
\end{equation}
While if $\alpha\ne Identity$ and we set
\begin{equation}\label{3.13}
\phi_{\gamma,\alpha}(t,s)=d_{\tilde g}(\tilde \gamma(t),\alpha(\tilde \gamma(s))), \quad |s|, |t|\le 1/2,
\end{equation}
then we can write for $|s|, |t|\le 1/2$
\begin{equation}\label{3.14}
K^\gamma_{T,\la,\alpha}(t,s)=
%\\
w(\tilde \gamma(t),\alpha(\tilde \gamma(s)))
\sum_\pm a_\pm(T,\la; \,  \phi_{\gamma,\alpha}(t,s))e^{\pm i\la \phi_{\gamma,\alpha}(t,s)}
+R^\gamma_{T,\la,\alpha}(t,s),
\end{equation}
where $w(x,y)$ is a smooth bounded function on $\Rtwo\times\Rtwo$ and where for each $j=0,1,2,\dots$ there is a constant
$C_j$ independent of $T,\la\ge1$ so that
\begin{equation}\label{3.15}
|\partial_r^ja_\pm(T,\la; \, r)|\le C_jT^{-1}\la^{\frac12}r^{-\frac12-j}, \quad r\ge 1,
\end{equation}
and for a constant $C_T$ which is independent of $\gamma$, $\alpha$ and $\la$
\begin{equation}\label{3.16}
|R_{T,\la,\alpha}^\gamma(t,s)|\le C_T, \quad \text{if } \, K^\gamma_{T,\la,\alpha}\not\equiv 0.
\end{equation}
Additionally,  for each $j=0,1,2,\dots$ 
and $k=0,1,2,\dots$ we have
 \begin{equation}\label{3.18}
  |\partial_t^j\partial_s^k \phi_{\gamma,\alpha}(t,s)|
 \le C_{T,j,k}, \quad \text{if } \, \alpha \ne Identity, \, \text{and } \, K^\gamma_{T,\la,\alpha}\not\equiv0.
 \end{equation}
\end{lemma}

%We shall use the estimates \eqref{3.15} for the main amplitude when we estimate $\Sstab$, while when we estimate $\Sosc$, we shall just
%use the cruder estimates \eqref{3.17} and \eqref{3.18}

\begin{proof}%[Proof of Lemma~\ref{lemma3.1}]
Given $x,y\in \Rtwo$ and $\tau\in \R$, $(\cos \tau\sqdt)(x,y)$ is smooth near $(\tau,x,y)$ if $d_{\tilde g}(x,y)\ne |\tau|$.  This implies \eqref{3.12}
since $(1-\beta(\tau))=0$ for $|\tau|\le 3/2$ and $d_{\tilde g}(\tilde \gamma(t),\tilde \gamma(s))\le 1$ if $|t|, |s|\le 1/2$.  Note also that
\begin{equation}\label{3.19}
d_{\tilde g}(\tilde \gamma(t),\alpha(\tilde \gamma(s)))\ge2 \quad \text{if } \, |t|, |s|\le 1/2 \, \, \text{and } \, \, \alpha\ne Identity
\end{equation}
since we are assuming that the injectivity radius of $(M,g)$ is larger than ten and $\alpha$ is a deck transformation.

To handle the nonidentity terms, let for $x\in \Rtwo$, $|x|\ge 1$,
\begin{align*}
K_0(|x|)&=\frac1{\pi T}\int_{\Rtwo}\int_{-\infty}^\infty \Hat \chi(\tau/T)e^{i\la \tau}\cos(\tau |\xi|) e^{ix\cdot \xi}\, d\tau d\xi
\\
&=\int_{\Rtwo}\chi(T(\la-|\xi|))e^{ix\cdot \xi}\, d\xi + \int_{\Rtwo}\chi(T(\la+|\xi|))e^{ix\cdot \xi}\, d\xi,
\end{align*}
and also let $\Phi_T(t)\in {\mathcal S}(\R)$ be defined by its Fourier transform, $\Hat \Phi_T(\tau)=\beta(\tau)\Hat \chi(\tau/T)$ and put
\begin{align*}
K_1(|x|)&=\frac1{\pi T}\int_{\Rtwo}\int_{-\infty}^\infty \beta(\tau) \Hat \chi(\tau/T)e^{i\la \tau}\cos(\tau |\xi|) e^{ix\cdot \xi}\, d\tau d\xi
\\
&=\frac1T\int_{\Rtwo}\Phi_T(\la-|\xi|)e^{ix\cdot \xi}\, d\xi + \frac1T\int_{\Rtwo}\Phi_T(\la+|\xi|)e^{ix\cdot \xi}\, d\xi.
\end{align*}
Recall that the Fourier transform of the induced Lebesgue measure on the unit circle in $\Rtwo$ is of the form
\begin{equation}\label{ftcirc}\widehat{d\theta}(y)=\int_{S^1}e^{i y\cdot (\cos \theta,\sin \theta)}\, d\theta
=|y|^{-\frac12}\sum_{\pm}a_\pm(|y|)e^{\pm i|y|}, \quad |y|\ge 1,
\end{equation}
where for each $j=0,1,2,\dots$
$$|\partial_r^j a_\pm(r)|\le C_jr^{-j}, \quad r\ge 1.$$
Also $\widehat{d\theta}\in C^\infty(\Rtwo)$ of course.  Thus, if $\la, T\ge1$, modulo a term
which is $O((\la|x|)^{-N}T^{-1})$ for any $N$ independent of $T$, we have
\begin{align}\label{3.20}
K_0(|x|)&=\frac1{|x|^{\frac12}}\int_0^\infty \chi(T(\la-r))\sum_{\pm}a_\pm(|x|r)e^{\pm ir|x|}r^{\frac12}\, dr
\\
&=\frac{\la^{\frac12}}{|x|^{\frac12}}\sum_\pm b_{\pm}(T,\la; |x|)e^{\pm i \la |x|},
\notag
\end{align}
where it is easy to check that
\begin{equation}\label{3.21}
|\partial_r^jb_\pm(T,\la;r)|\le C_jT^{-1}r^{-j}, \quad r, \la, T\ge 1, \, \, j=0,1,\dots.
\end{equation}
%Thus,
%\begin{equation}\label{3.22}
%|K_0(|x|)|\le CT^{-1}\la^{\frac12}|x|^{-\frac12}, \quad |x|, \la, T\ge 1.
%\end{equation}
Similar arguments show that, modulo an $O((\la|x|)^{-N}T^{-1})$ error we also have
\begin{equation}\label{3.23}
K_1(|x|)=\frac{\la^{\frac12}}{|x|^{\frac12}}\sum_\pm \tilde b_\pm(T,\la,|x|)e^{\pm i\la |x|}, %\quad r, T,\la\ge 1, \, \, 
\end{equation}
with
\begin{equation}\label{3.24}
|\partial_r^j\tilde b_\pm(T,\la, r)|\le C_jT^{-1}r^{-j}, \quad r, \la, T\ge 1,  \, \, j=0,1,2\dots.
\end{equation}
%and so in particular we also have
%\begin{equation}\label{3.25}
%|K_1(|x|)|\le CT^{-1}\la^{\frac12}|x|^{-\frac12}, \quad \la, T\ge 1, \, \, |x|\ge 2.
%\end{equation}

To use these stationary phase calculations and obtain the remaining parts of the lemma, \eqref{3.13}--\eqref{3.18}, as we did in the last section, we shall use the Hadamard parametrix.  This time, though, we naturally will use it for $(\Rtwo, \tilde g)$.  For $x,y\in \Rtwo$ we now note that we can write
\begin{multline}\label{3.26}
\bigl(\cos \tau \sqdt\bigr)(x,y)=(2\pi)^{-2}w(x,y)\int_{\Rtwo}e^{i d_{\tilde g}(x,y)\xi_1 } \cos(\tau |\xi|)\, d\xi
\\
+\sum_{\pm}\int_{\Rtwo}e^{id_{\tilde g}(x,y)\xi_1} e^{\pm i\tau|\xi|} a_{\pm}(\tau, x,y,|\xi|) \, d\xi +R(\tau,x,y),
\end{multline}
where we can take the ``remainder'' to satisfy
$$|R(\tau,x,y)|\le C_T, \quad \text{if } \, |\tau|\le 2T,$$
and $a_\pm$ is a symbol of order two which in particular satisfies
$$|\partial^j_\tau a_\pm(\tau, x,y, |\xi|)|\le C_{T,j}(1+|\xi|)^{-2}, \quad \text{if } \, |\tau|\le 2T, \, \, \, j=0,1,2,\dots,$$
and where the leading coefficient $w$ is smooth and nonnegative and satisfies
$$w(x,y)\le C$$
independent of $x,y\in \Rtwo$, by volume comparison theorems (see \cite{Berard} and \cite{SZ3}).

Clearly if we replace $\bigl(\cos \tau \sqdt\bigr)(\tilde \gamma(t), \alpha(\tilde \gamma(s)))$ by $R$ in \eqref{3.11} we will obtain something which
is bounded by a constant which will depend on $T$ but can be taken to be independent of $\la>1$.  In the previous section we argued that
the contribution of the term involving the symbol of order $-2$ leads to a bounded quantity.  Now since we are working in two-dimensions
instead of three-dimensions, we can repeat those arguments to see that this term leads to something that is actually
$\le C_T\la^{-1}$ if $\la, T\ge 1$, which is much better than we need for the main term.

If we take $x=\tilde \gamma(t)$ and $y=\alpha(\tilde \gamma(s))$, $|t|, |s|\le 1/2$,  for the first term in the right side of \eqref{3.26} and replace the
cosine-transform kernel in \eqref{3.11} by this expression then we will exactly obtain $(2\pi)^2$ times
$$w(\tilde \gamma(t),\alpha(\tilde \gamma(s)))\, K_0(d_{\tilde g}(\tilde \gamma(t),\alpha(\tilde \gamma(s))) - w(\tilde \gamma(t),\alpha(\tilde \gamma(s)))\, K_1(d_{\tilde g}(\tilde \gamma(t), \alpha(\tilde \gamma(s))).$$
As a result, \eqref{3.13}-\eqref{3.16} follows from \eqref{3.19}-\eqref{3.23}.  
Since each deck transformation is smooth,
the remaining part of the lemma, \eqref{3.18},
follows from the above, the chain rule, \eqref{3.19} and the fact that the Riemannian distance
function is smooth away from the diagonal.
\end{proof}

Let us now use this lemma to estimate $\Sstab$ defined in \eqref{3.9}.  If $\alpha_1, \alpha_2\in \Stab$ and $\alpha_1\ne \alpha_2$ then the two geodesic segments $\tilde \gamma_1 = \{\alpha_1(\tilde \gamma(t)): \, |t|\le 1/2\}$ and $\tilde \gamma_2 = \{\alpha_2(\tilde \gamma(t)): \, |t|\le 1/2\}$ are disjoint since the injectivity radius of $(M,g)$ is ten or more and the covering map is defined by the exponential map at $\gamma(0)$.  Since $\Stab$ fixes the first coordinate axis and is an isometry, it follows that $\tilde \gamma_1$ and $\tilde \gamma_2$ are
disjoint unit length segments of this axis.  Since $\Hat \chi(\tau/T)=0$ if $|\tau|\ge T$, by \eqref{3.8}, the summands in \eqref{3.9} vanish when 
$$\{\alpha(\tilde \gamma(s)): \, |s|\le 1/2\}\cap \{(t,0): \, |t|\le T+1\}=\emptyset.$$
Using these facts along with \eqref{3.12} and \eqref{3.14}-\eqref{3.16}, we conclude that for $|t|, |s|\le1/2$ the kernel $K^{\text{Stab}}_\la$ of $\Sstab$ satisfies
\begin{equation}\label{kest}
|K^{\text{Stab}}_\la(t,s)|\le CT^{-1}\la^{\frac12}\sum_{0\le j\le T+1}(1+j)^{-\frac12}+C_T\le C\la^{\frac12}T^{-\frac12}+C_T.
\end{equation}
Consequently, 
\begin{equation}\label{3.27}
\|\Sstab h\|_{L^4([-\frac12,\frac12])}\le \bigl(C\la^{\frac12}T^{-\frac12}+C_T\bigr)\, \|h\|_{L^{\frac43}([-\frac12,\frac12])},
\end{equation}
which is better than the bounds posited in \eqref{3.4} for the whole operator.

Based on this, in order to obtain \eqref{3.4} it is enough to show that for the remaining piece of $S_\la$ we have
\begin{equation}\label{3.28}
\|\Sosc  h\|_{L^4([-\frac12,\frac12])}\le C_T\la^{\frac38} \|h\|_{L^{\frac43}([-\frac12,\frac12])}.
\end{equation}
To do this, we shall show using \eqref{3.14} and the following lemma that, modulo a trivial error, each of the nonzero summands in \eqref{3.10} is a well-behaved oscillatory integral operator.

\begin{lemma}\label{lemma3.2}  Let $\gamma_1=\{ \gamma_1(t), t\in \R\}$ and $\gamma_2=\{\gamma_2(s), s\in \R\}$ be two {\em distinct} geodesics,
each parameterized by arclength,  on ${\mathbb R}^2$ endowed with a Riemannian metric $\tilde g$ of everywhere nonpositive curvature.     If $d_{\tilde g}$ denotes the Riemannian distance function associated to the metric put
$\phi(t,s)=d_{\tilde g}(\gamma_1(t),\gamma_2(s))$.
It then follows that
\begin{equation}\label{a}
\phi''_{ts}(t,s)\ne 0, \, \, \, \text{for all } \, \, t,s\in \R, \quad \text{if } \, \, \gamma_1\cap \gamma_2=\emptyset.
\end{equation}
If $\gamma_1\cap \gamma_2\ne \emptyset$ then there is a unique point $x_0(\gamma_1,\gamma_2)$ such that $\gamma_1\cap \gamma_2=\{x_0(\gamma_1,\gamma_2)\}$.
Then if $t_0$ and $s_0$ are the unique numbers such that $\gamma_1(t_0)=x_0(\gamma_1,\gamma_2)$ and $\gamma_2(s_0)=x_0(\gamma_1,\gamma_2)$, it follows that
\begin{equation}\label{b}
\phi''_{st}(t,s)\ne 0 \quad \text{if } \, \, t\ne t_0 \, \, \text{and } \, s\ne s_0.
\end{equation}
Furthermore, the map
\begin{equation}\label{c}
(t,s)=\Pi(t,s)=(t,\phi'_t(t,s))
\end{equation}
has a folding singularity when $t\in \R\backslash \{t_0\}$ and $s=s_0$.
\end{lemma}

\begin{proof}  By our curvature assumptions and a theorem of Hadamard (see, e.g. \cite{do Carmo}), for every point $p\in {\mathbb R}^2$, the exponential map associated
with $\tilde g$  from the tangent space at $p$ to ${\mathbb R}^2$ is a diffeomorphism.  Hence, $\gamma_1$ and $\gamma_2$ can intersect at most one point.

To prove the facts about the mixed Hessian of $\phi$, we may assume that $p\in \gamma_1$ and we may work in geodesic normal coordinates about $p$ so that
$\gamma_1$ is the $x_1$-axis, i.e.,
$$\gamma_1(t)=(t,0), \quad t\in \R.$$
We then write the second geodesic
$$\gamma_2(s)=(x_1(s),x_2(s))$$
in these coordinates.

If for a given $s$, $ \gamma_2(s)\notin \gamma_1$, it follows that $x_2(s)\ne 0$.  We shall show that if
$$\phi(t,s)=d_{\tilde g}\bigl((t,0),(x_1(s),x_2(s)\bigr),$$
then
\begin{equation}\label{d}
\phi''_{ts}(0,s)\ne 0, \quad \text{if } \, \, x_2(s)\ne 0 \, \, \text{and } \, (0,0)\notin \gamma_2,
\end{equation}
which implies all but the last part of the lemma since $p$ was chosen to be an arbitrary point on $\gamma_1$.

To prove \eqref{d}, since we are working in geodesic normal coordinates about the origin, we have
\begin{equation}\label{e}
%\frac{\partial \phi(0,s)}{\partial t}
\phi_t'(0,s)=
=\frac{-x_1(s)}{\sqrt{x_1^2(s)+x_2^2(s)}}.
\end{equation}
If $x_1(s)=0$ then $\dot x_1(s)=dx_1(s)/ds\ne 0$, since otherwise $\dot x(s)=(0,\pm 1)$, which would mean that $\gamma_2$ is the geodesic
$\{(0,s), \, s\in \R\}$ and contradict our assumption in \eqref{d} that $(0,0)\notin \gamma_2$.  Note that the $x_2$-axis is a geodesic since we are working in geodesic normal coordinates
about the origin.

Since we are assuming for now that $x_2(s)\ne 0$, we see from \eqref{e} that $\phi''_{ts}(0,s)\ne 0$ if $x_1(s)=0$ and $(0,0)\notin \gamma_2$.  The remaining case that we need
to consider in order to establish \eqref{a} and \eqref{b} is the case where $x_1(s)\ne 0$, $x_2(s)\ne 0$ and, again, $(0,0)\notin \gamma_2$.  In this case, by \eqref{e}, we can write
$$\phi'_t(0,s)=\frac{\pm 1}{\sqrt{1+(x_2(s)/x_1(s))^2}},$$
and so
$$\phi''_{ts}(0,s)=\frac{\pm x_2(s)/x_1(s)}{[1+(x_2(s)/x_1(s))^2]^{3/2}}\times \frac{d}{ds}\Bigl(\, \frac{x_2(s)}{x_1(s)}\, \Bigr).$$
Therefore, under our current assumption that $x_1(s)\ne 0$ and $x_2(s)\ne 0$, we have
$$\phi''_{st}(0,s)=0 \, \, \iff \, \, \frac{x_1(s)\dot x_2(s)-\dot x_1(s)x_2(s)}{x_1^2(s)}=0.$$
If the last term vanished then one of $\pm(\dot x_1(s), \dot x_2(s))$ would have to point in the same direction as $(x_1(s),x_2(s))$, meaning that $\gamma_2$ is the geodesic which is the
line through the origin and $(x_1(s),x_2(s))$, which is ruled out by our assumption that $(0,0)\notin \gamma_2$ in \eqref{d}.  This completes the proof of \eqref{d}, which finishes the proof
of all parts of the lemma except for the assertion about the map $\Pi$ in \eqref{c}.

To prove the last part of the lemma (which is just a manifestation of the Gauss lemma in Riemannian geometry), we may assume that $\gamma_1(0)\notin \gamma_1\cap \gamma_2$ and, as before, work in geodesic normal coordinates about $\gamma_1(0)$
so that $\gamma_1(t)=(t,0)$ and, in particular $\gamma_1(0)$ is the origin.  We then must show that if $\gamma_2(s)=(x_1(s),x_2(s))$ is as before then the map
\begin{equation}\label{f}
(t,s)\to \Pi(t,s)=(t,\phi'_t(t,s))
\end{equation}
has a fold singularity at $(0,s_0)$ if
\begin{equation}\label{g}
x_2(s_0)=0, \, \, \, x_1(s_0)\ne 0, \, \, \text{and } \, \dot x_2(s_0)\ne 0.
\end{equation}
Note that the assumption that $\dot x_2(s_0)\ne 0$ is equivalent to our assumption that $\gamma_1\ne \gamma_2$ if $\gamma_2(s_0)=(x_1(s_0),0)$.

By what we just did we have that
$$\phi''_{ts}(t,s_0)\equiv 0 \quad \text{for } \, \, t \, \, \text{near } \, \, 0,$$
and if $d\Pi$ denotes the differential of the map $\Pi$, i.e.,
\begin{equation}\label{h}
d\Pi(t,s)=\left(\begin{array}{cc}
1 & 0
\\ \\
\phi''_{tt}(t,s) & \phi''_{ts}(t,s)
\end{array}\right),
\end{equation}
then, by \eqref{b}, 
$$\det \bigl(d\Pi(t,s)\bigr) = \phi''_{ts}(t,s)\ne 0, \quad \text{if } \, \, s\ne s_0 \, \, \text{and } \, \, t \, \, \text{near } \, \, 0.$$
In other words, the singular variety $\Sigma$ of all $(t,s)$ $t$ near $0$ and $s\in \R$ where $\det(d\Pi(t,s))=0$ is the set of pairs where we have $s=s_0$.  By \eqref{h}, the kernel
of $d\Pi$ if $(t,s)\in \Sigma$ is the one-dimensional space spanned by $(0,1)$, which means that the kernel of $d\Pi$ is transversal to $\Sigma$.  
As a result, 
$(t,s)\to \Pi(t,s)=(t,\phi'_t(t,s))$ has a fold singularity at $(0,s_0)$ if and only if
\begin{equation}\label{i}
\phi'''_{tss}(0,s_0)\ne 0.
\end{equation}
By \eqref{e}, for $s$ near $s_0$ the map
\begin{equation}\label{j}
s\to \frac{\partial \phi(0,s)}{\partial t}=-\cos \theta(s),
\end{equation}
if
$$\bigl(\cos\theta(s),\sin \theta(s)\bigr)=\frac{(x_1(s),x_2(s))}{\sqrt{x_1^2(s)+x_2^2(s)}}$$
denotes the projection of the curve $\gamma_2(s)=(x_1(s),x_2(s))$ onto the unit circle.  Since we are assuming that $\dot x_2(s_0)\ne 0$, we have that
\begin{equation}\label{k}
\dot \theta(s_0)\ne 0.
\end{equation}
Because $\gamma_2(s_0)=(x_1(s_0),0)$ and $x_1(s_0)\ne 0$, it follows that $\cos \theta(s_0)=\pm 1$, and therefore we get \eqref{i} from the chain rule if we use
\eqref{j} and \eqref{k}.
\end{proof}

To prove \eqref{3.28} we need to recall some oscillatory integral estimates for operators of the form
\begin{equation}\label{3.40}
\bigl(T_\la f\bigr)(t)=\int_{-\infty}^\infty e^{i \la \varphi(t,s)}\, a(t,s) f(s)\, ds,
\end{equation}
where we assume that $a, \varphi\in C^\infty(\R\times \R)$ and that $\varphi$ is real valued.  H\"ormander~\cite{Hor} proved that for $\la\ge1$ one has
\begin{equation}\label{3.41}
\|T_\la f\|_{L^q(\R)}\le C\la^{-\frac1q}\|f\|_{L^{\frac{q}{q-1}}(\R)} \quad \text{if } \, q\ge 2, \, \, \text{and } \, \, \varphi''_{ts}\ne 0 \, \, \text{on } \, \, \text{supp }a.
\end{equation}
The hypothesis in H\"ormander's oscillatory integral theorem is equivalent to the statement that if
$${\mathcal C}_\varphi =\bigl\{(t,\varphi'_t,s,\varphi'_s)\bigr\}\subset T^*\R\times T^*\R$$
is the canonical relation associated with the phase then the left projection
$$\Pi_L(t,s)=(t,\varphi'_t)$$
{\em and} the right projection
$$\Pi_L(t,s)=(s,\varphi'_s)$$
are local diffeomorphisms on $\text{supp }a$.

Pan and the second author~\cite{PanS} studied the next most straightforward case  where the hypothesis is relaxed and one assumes that {\em both} maps
are submersions with fold singularities.  In this case the power $\frac12$ in \eqref{3.41} has to be replaced by $\frac13$, which is sharp.  As was shown in
\cite{burq} and \cite{Hu} this result is relevant for studying restriction problems associated to curves possessing nonvanishing curvature.

Lemma~\ref{lemma3.2} (and earlier arguments in \cite{burq} and \cite{Hu}) indicate that, for the problem at hand of obtaining bounds for restrictions to geodesics, it is
useful to have estimates for the oscillatory integrals in \eqref{3.40} when just one of the two projections is a fold at the singular variety.  Greenleaf and Seeger~\cite{GS}
and Phong and Stein~\cite{PS0} showed that if one only assumes that $\Pi_L$ has fold singularities (and no assumptions about $\Pi_R$) then
$$\|T_\la f\|_{L^q(\R)}\le C\la^{-\frac14}\|f\|_{L^2(\R)}, \quad 2\le q\le 4.$$
We point out that unlike the case where one of $\Pi_L$ or $\Pi_R$ being nonsingular implies the same for the other projection, it need not follow that one of $\Pi_L$
or $\Pi_R$ is a submersion with folds if the other is.  Indeed the last part of Lemma~\ref{lemma3.2} provides an example where this does not happen.

Since, on account of Lemma~\ref{lemma3.2}, we need to deal with the situation where either $\Pi_L$ or $\Pi_R$ has a fold singularity (but no assumption on
the other), we shall just use the preceding estimate when $q=2$.  By taking adjoints we see that $\|T_\la\|_{L^2(\R)\to L^2(\R)}=O(\la^{-\frac14})$ if we just assume that
one of the maps $\Pi_L$ or $\Pi_R$  has fold singularities.  By interpolating with the trivial $L^1(\R)\to L^\infty(\R)$ bounds we conclude from this that
\begin{multline}\label{fold}
\|T_\la f\|_{L^q(\R)}\le C \la^{-\frac1{2q}}\|f\|_{L^{\frac{q}{q-1}}(\R)}, \, \, \, \text{if }q\ge 2,
\\
\text{and either} \, \, \Pi_L \, \, \text{or } \, \, \Pi_R\, \, \text{is a submersion with folds.}
\end{multline}
The constant $C$ here depends only on the size of finitely many derivatives of $\varphi$ and $a$ and on $\text{supp }a$ and on the lowerbounds for
$|\varphi''_{ts}|+|\varphi'''_{tts}|+|\varphi'''_{tss}|$ on this set.  Similarly, the one in H\"ormander's estimate \eqref{3.41} just depends on the size of finitely many
derivatives of the phase and amplitude and lower bounds for $|\varphi''_{ts}|$ on $\text{supp }a$.

Lemma~\ref{lemma3.2} tells us that for each of the phase functions in \eqref{3.13} then we either have the hypothesis in \eqref{3.41} or the one in \eqref{fold} whenever
$\alpha\notin \Stab$  (of course we get neither hypothesis if $\alpha\in \Stab$).  Therefore by \eqref{3.41} and \eqref{fold} for $q=4$ along with \eqref{3.14}--\eqref{3.18}, we have
that if $\alpha\notin \Stab$ then there is a constant $C_\alpha$ such that
\begin{equation}\label{3.43}
\Bigl\|\int_{-\frac12}^{\frac12} K^\gamma_{T,\la,\alpha}(t,s) h(s)\, ds\Bigr\|_{L^4([-\frac12,\frac12])}\le C_\alpha \la^{\frac38}\|f\|_{L^{\frac43}([-\frac12,\frac12])}, \quad \la>1.
\end{equation}
If we sum up over the $O(e^{CT})$ nonzero terms in \eqref{3.10}, we conclude that \eqref{3.28} is valid, which finishes the proof of the estimate \eqref{3.4} associated with
our fixed element $\gamma$ of the the space $\Pi$ of all unit length geodesics.

To prove the uniform bounds as $\gamma$ ranges over $\Pi$, let us fix an element $\gamma_0\in \Pi$ and let $\Sstab$ be the operator associated to
elements of the stabilizer group associated to $\gamma_0$.  Clearly, we then have \eqref{3.27} with uniform bounds as $\gamma$ ranges over
a neighborhood of $\gamma_0$ in $\Pi$.
 Also, if $\alpha$ is not in
the stabilizer group for $\gamma_0$ then it is not in the stabilizer group for nearby elements $\gamma$ of $\Pi$.  Based on this we claim that, if $\Sosc$ is the operator associated to $\gamma_0$,  then there exists a neighborhood ${\mathcal N}(\gamma_0)$ of $\gamma_0$ in $\Pi$ so that \eqref{3.28} holds with uniform constants as
$\gamma$ ranges over ${\mathcal N}(\gamma_0)$.  We use the fact that this neighborhood can be chosen so that the constants in 
%\eqref{3.17}--
\eqref{3.18} can be chosen to be independent of the elements of ${\mathcal N}$ for any finite collection of  $(j,k)$.  Thus, it  follows from \eqref{3.41} and 
\eqref{fold} (and the dependence of their constants) that we also have \eqref{3.43} with uniform constants for elements of ${\mathcal N}(\gamma_0)$.  This means that 
the
constant in \eqref{3.28} can be chosen to be independent of the elements of ${\mathcal N}(\gamma_0)$ as well, and, by the compactness of $\Pi$, we conclude from this
that the constant in \eqref{3.28} can also be chosen to be independent of $\gamma\in \Pi$.  Thus, \eqref{3.4} is valid with constants independent of $\gamma\in \Pi$,
which in turn yields \eqref{3.2} and completes the proof of Theorem~\ref{theorem2}.  \qed

%The proof of \eqref{3.27} clearly shows how the constant there can be taken to be independent of $\gamma\in \Pi$.  Since, given $\gamma\in \Pi$  and a 
%finite set in $({\mathbb N} \cup \{0\})\times ({\mathbb N}\cup \{0\})$ there is a neighborhood ${\mathcal N}$ of $\gamma$ in $\Pi$ so that the constants in 
%%\eqref{3.17}--
%\eqref{3.18} can be chosen to be independent of the elements of ${\mathcal N}$ when $(j,k)\in E$, it also follows from \eqref{3.41} and 
%\eqref{fold} (and the dependence of their constants) that we also have \eqref{3.43} with uniform constants for elements of ${\mathcal N}$.  This means that 
%the
%constant in \eqref{3.28} can be chosen to be independent of the elements of ${\mathcal N}$ as well, and, by the compactness of $\Pi$, we conclude from this
%that the constant in \eqref{3.28} can also be chosen to be independent of $\gamma\in \Pi$.  Thus, \eqref{3.4} is valid with constants independent of $\gamma\in \Pi$,
%which in turn yields \eqref{3.2} and completes the proof of Theorem~\ref{theorem2}.  \qed

%%%%%%%%%%%%%%%%%%%

\newsection{Further improvements in 3-dimensions}

The purpose of this section is to show that we can obtain further improvements of the known 3-dimensional restriction estimates for geodesics if
we assume {\em constant} nonpositive curvature.

To prove Theorem~\ref{theorem3}, we require an analog of Lemma~\ref{lemma3.2} for the 3-dimensional case which is more difficult than its
2-dimensional counterpart.  Our proof requires the assumption of constant curvature, but it would be interesting if this assumption could be weakened or removed.  This result and Theorem~\ref{theorem3} are somewhat reminiscent of Nikodym maximal estimates obtained by one of us
in \cite{SoJAMS}, whose proof also required the assumption of constant curvature.  In that work and the present, it is very convenient to be in
the position where there are many totally geodesic submanifolds, which is ensured by the assumption of constant curvature.

If we are willing to make this assumption then we have the following analog of Lemma~\ref{lemma3.2}.

\begin{lemma}\label{lemma4.1}  Let $\gamma_1=\{\gamma_1(t): t\in \R\}$ and $\gamma_2=\{\gamma_2(s): s\in \R\}$ be two {\em distinct} geodesics each parameterized by arclength
in ${\mathbb R}^3$ endowed with a metric $\tilde g$ of constant nonpositive sectional curvature.  If $d_{\tilde g}$ denotes the Riemannian distance function associated to the metric put
$\phi(t,s)=d_{\tilde g}(\gamma_1(t),\gamma_2(s))$.  It then follows that there is at most one point $(t_0,s_0)$ so that
\begin{equation}\label{4.1}
|\phi''_{ts}(t,s)|+|\phi'''_{tss}(t,s)|+|\phi'''_{tts}(t,s)|\ne 0, \quad \text{if } \, \phi(t,s)\ne0, \, \, \, \text{and } \, \, (t,s)\ne (t_0,s_0).
\end{equation}
\end{lemma}

\begin{proof}
Since we are assuming that that $({\mathbb R}^3,\tilde g)$ has constant nonpositive sectional curvature it follows that the curvature must be either zero or negative.  In the case of zero curvature, $\phi(t,s)$ just denotes the Euclidean distance between points two distinct lines.  Since this case is very straightforward and follows from the argument for constant negative curvature, we shall assume that the sectional curvature of $({\mathbb R}^3,\tilde g)$ is identically equal to $-\kappa^2$ for some $\kappa>0$.

As in the proof of Lemma~\ref{lemma3.2}, we shall work in geodesic normal coordinates which vanish at a point $p\in \gamma_1$ satisfying $p\notin\gamma_2$.  We may assume further, as before, that, in these coordinates, $\gamma_1$ is the first coordinate axis, i.e., 
$$\gamma_1(t)=(t,0,0), \quad t\in \R.$$  For later use, we record that because of our assumption about constant sectional curvature, the metric in geodesic normal coordinates is given
by the formula \begin{equation}\label{4.2}g_{jk}(x)=\frac{x_jx_k}{|x|^2}+\frac{\sinh^2(\kappa |x|)}{\kappa^2|x|^2}\Bigl\{\delta_{jk}-\frac{x_jx_k}{|x|^2}\Bigr\},\end{equation}
with $|x|$ denoting the Euclidean distance to the origin.  (See \cite{Chavel}, Theorem II.8.2)

Let us write the second geodesic in these coordinates as
$$\gamma_2(s)=(x_1(s),x_2(s),x_3(s)), \quad s\in \R.$$
Then if
$$\phi(t,s)=d_{\tilde g}\bigl((t,0,0), (x_1(s),x_2(s),x_3(s))\bigr),$$
we need to verify \eqref{4.1}.  
%Since $p$ was chosen to be an arbitrary point in $\gamma_1\backslash \gamma_2$, it suffices to show that there is at most one
%value $s_0$ of $s$ so that
%\begin{equation}\label{con3}
%|\phi_{ts}''(0,s)|+|\phi'''_{tss}(0,s)|+|\phi'''_{tts}(0,s)|\ne 0, \quad \text{if} \,  \, \phi(0,s)\ne 0, \, \, \, \text{and } \, s\ne s_0.
%\end{equation}

To simplify the calculation, after relabeling, we may assume that we are calculating the  Hessian of $\phi$ and its derivatives at $t=s=0$ and
that $\gamma_2(0)\ne (0,0,0)$.
Furthermore,
after perhaps performing a rotation fixing the first coordinate axis, we may also assume that
$$\gamma_2(0)=(a,b,0), \quad (a,b)\ne (0,0).$$
If $b=0$ then $\gamma_2(0)$ is on $\gamma_1$, and in this case, the proof of Lemma~\ref{lemma3.2} can be adapted to show that $\phi'''_{tss}(0,0)\ne 0$.  Therefore,
we shall assume that $x_2(0)\ne 0$ in what follows.

Then
$$\phi''_{ts}(0,0)=\frac{\dot x_1(0)}{\sqrt{x_1^2(0)+x_2^2(0)+x_3^2(0)}}-\frac{x_1(0)(x_1(0)\dot x_1(0)+x_2(0)\dot x_2(0)+x_3(0)\dot x_3(0))}{(x_1^2(0)+x_2^2(0)+x_3^2(0))^{\frac32}}.
$$
From this we deduce that if $\phi''_{ts}(0,0)=0$ then we must have
$$0=x_2(0)\bigl(\dot x_1(0)x_2(0)-x_1(0)\dot x_2(0)\bigr).$$
Since we are assuming that $x_2(0)\ne 0$
%, due to the fact that $\gamma_2(0)$ is assumed to be off the $x_1$-axis, 
this means that if $\phi''_{ts}(0,0)=0$ then we must have
$$\dot \gamma_2(0)=\alpha\bigl(x_1(0),x_2(0),x_3(0)\bigr) +\beta(0,0,1),$$
for some $\alpha$ and $\beta$ in $\R$.  If $\beta=0$, then $\gamma_2$ must be the geodesic which is the line passing through $(a,b,0)$ and the origin.  In this case, one can
adapt the proof of Lemma~\ref{lemma3.2} to see that $\phi'''_{tts}(0,0)\ne 0$.

Based on this, it suffices to handle the case where
\begin{multline}\label{4.3}
\gamma_2(0)=(a,b,0)\ne (0,0,0), \quad
\, \dot\gamma_2(0)=\alpha(a,b,0)+\beta(0,0,1), \quad \beta\ne 0, \\ \text{and } \, \phi''_{ts}(0,0)=0.
\end{multline}
Since our assumption that $\phi''_{ts}(0,0)$ leads to
\begin{align*}
\phi&'''_{tss}(0,0)
\\
&=\bigl(x_1^2(0)+x_2^2(0)+x_3^2(0)\bigr)^{-\frac32}
\\
&
\times \frac{d}{ds}\Bigl(
\dot x_1(s)(x_1^2(s)+x_2^2(s)+x_3^2(s))-x_1(s)(x_1(s)\dot x_1(s)+x_2(s)\dot x_2(s)+x_3(s)\dot x_3(s))\Bigr) \Bigl|_{s=0}
\\
&=\bigl(x_1^2(0)+x_2^2(0)+x_3^2(0)\bigr)^{-\frac32}
\, \Bigl[x_2(0)\bigl(\ddot x_1(0)x_2(0)-x_1(0)\ddot x_2(0)\bigr)-x_1(0)\dot x_3^2(0)\Bigr],
\end{align*}

Note that $\ddot x_1(0)x_2(0)-x(0)\ddot x_2(0)$ is a multiple of the inner product of $\ddot\gamma_2(0)$ with a unit normal vector to plane spanned by the radial
vector $(a,b,0)$ and the vertical vector $(0,0,1)$.  Since our assumption \eqref{4.3} implies that $\gamma_2(0)$ and $\dot \gamma_2(0)$ belong to this plane,
it would follow that
\begin{equation}\label{4.4}\ddot x_1(0)x_2(0)-x_1(0)\ddot x_2(0)=0\end{equation}
if this plane were totally geodesic.  In view of \eqref{4.2} and Jacobi's equation for geodesics, this turns out to be the case due to the fact that, if there is constant negative sectional curvature, then, in geodesic
normal coordinates, all two-planes passing through the origin are totally geodesic.

Since we are assuming that $\dot x_3(0)=\beta\ne 0$, \eqref{4.4} and the above calculation implies that, under the assumptions in \eqref{4.3} we have
\begin{equation*}%\label{4.5}
\phi'''_{tss}(0,0)=0 \, \iff \, x_1(0)=0.
\end{equation*}
In other words,
$$\gamma_2(0)=(0,b,0), \, \, b\ne 0.$$
This means that if $\phi''_{ts}(0,0)=\phi'''_{tss}(0,0)=0$ then the geodesic $\gamma(\tau)$ connecting $\gamma_1(0)$ and $\gamma_2(0)$ must 
satisfy
$$\langle \dot\gamma(\tau_1), \dot \gamma_1(0)\rangle_{\tilde g}=0, \quad \text{when } \, \gamma(\tau_1)=0.$$
Put more succinctly, the geodesic connecting $\gamma_1(0)$ and $\gamma_2(0)$ must be orthogonal to $\gamma_1$.
  Interchanging the roles of $t$ and $s$, we conclude that if $\phi''_{ts}(0,0)=\phi'''_{tss}(0,0)=\phi'''_{tts}(0,0)=0$ then this connecting geodesic
must be orthogonal to {\em both} $\gamma_1$ and $\gamma_2$.

If this is the case, we conclude that when $s_1\ne 0$ and $t_1\ne 0$ and $\phi(t_1,s_1)\ne 0$, then we cannot have
$\phi''_{ts}(t_1,s_1)=\phi'''_{tss}(t_1,s_1)=\phi'''_{tts}(t_1,s_1)=0$.  For if all three quantities vanished then the geodesic
connecting $\gamma_1(t_1)$ and $\gamma_2(s_1)$ would have to be orthogonal to both $\gamma_1$ and
$\gamma_2$.  Therefore the geodesic quadrilateral consisting of this connecting geodesic along with the one connecting
$\gamma_1(0)$ and $\gamma_2(0)$ and the geodesic segment of $\gamma_1$ between $\gamma_1(0)$ and
$\gamma_1(t_1)$ and the geodesic segment of $\gamma_2$ between $\gamma_2(0)$ and $\gamma_2(s_1)$
would have to be a quadrilateral of exactly 360 degrees, which is impossible, since in manifolds of negative sectional curvature
quadrilaterals always have total angles of {\em less} than 360 degrees.  Since a similar argument based on the fact that,
in this setting, geodesic triangles have total angle less than 180 degrees shows that if $\phi''_{ts}(0,0)=\phi'''_{tss}(0,0)=\phi'''_{tts}(0,0)=0$
then $|\phi''_{ts}(0,s)|+|\phi'''_{tss}(0,s)|+|\phi'''_{tts}(0,s)|\ne 0$ if $s\ne 0$ and $\phi(0,s)\ne 0$, the proof is complete.
\end{proof}

\begin{proof}[Proof of Theorem~\ref{theorem3}]
As in the proof of Theorem~\ref{theorem2}, let us first show that under the assumption of constant nonpositive curvature we can obtain improvements
for restriction estimates for a fixed unit length geodesic $\gamma \in \Pi$.  Specifically, we shall use Lemma~\ref{lemma4.1} to show that given
$\e>0$ we can find $\Lambda(\e,\gamma)<\infty$ so that
\begin{equation}\label{4.5}
\Bigl(\int_\gamma |e_\la|^2\, ds \Bigr)^{\frac12}\le \e \la^{\frac12}\|e_\la\|_{L^2(M)}, \quad \la \ge \Lambda(\e, \gamma).
\end{equation}

If we choose $\rho \in {\mathcal S}(\R)$ as in the proof of Theorem~\ref{theorem2}, we can repeat the first part of the proof of that result to see
that \eqref{4.5} would follow from showing that for $T\gg 1$ fixed we have
\begin{multline}\label{4.6}
\Bigl(\int_\gamma \bigl|\rho(T(\la-\sqdt))f\bigr|^2 \, ds \Bigr)^{\frac12}
\le C(1+\la)^{\frac12}\bigl(T/\log(2+T)\bigr)^{-\frac12}\|f\|_{L^2(M)}
\\
+C_T(1+\la)^{\frac38}\|f\|_{L^2(M)},
\end{multline}
where the constant $C$ here is independent of $T$, unlike $C_T$.

Repeating the $TT^*$ argument showing how \eqref{3.3} follows from \eqref{3.2}, we conclude that if $\chi(\tau)=(\rho(\tau))^2$, and
if $\gamma=\gamma(t)$, $|t|\le 1/2$, is a parameterization of $\gamma$ by arc length of our fixed $\gamma\in \Pi$, then we would obtain
\eqref{4.6} if we could show that for $\la>1$ we have
\begin{multline}\label{4.7}
\Bigl(\int_{-\frac12}^{\frac12}\Bigl|\, \int_{-\frac12}^{\frac12}
\chi\big(T(\la-\sqdt)\Bigr)(\gamma(t),\gamma(s)) \, h(s)\, ds\, \Bigr|^2 \, dt\Bigr)^{\frac12}
\\
\le C\la (T/\log(2+T))^{-1}\|h\|_{L^2([-\frac12, \frac12])} +C_T\la^{\frac34}\|h\|_{L^2([-\frac12,\frac12])}.
\end{multline}

We may assume, as before, that the injectivity radius of $(M,g)$ is ten or more.  Then, as in the proof of Theorem~\ref{theorem2}, we pick a bump
function $\beta\in C_0^\infty(\R)$ satisfying
$$\beta(\tau)=1 \, \, \, \text{for } \, \, |\tau|\le 3/2, \, \, \text{and } \, \, \beta(\tau)=0, \, \, \, |\tau|\ge 2.$$
We then may write
\begin{multline*}
\chi(T(\la-\sqd))(x,y)=\frac1{2\pi T}\int \beta(\tau)\Hat \chi(\tau/T) e^{i\la\tau} \bigl(e^{- i\tau\sqd}\bigr)(x,y)\, d\tau
\\
+\frac1{2\pi T}\int (1-\beta(\tau)) \, \Hat \chi(\tau/T) e^{i\la\tau} \bigl(e^{- i\tau\sqd}\bigr)(x,y)\, d\tau
=K_0(x,y)+K_1(x,y).
\end{multline*}

Since we are assuming that the injectivity radius is ten or more and $\beta(\tau)=0$ for $|\tau|\ge2$ it follows from the proof of Theorem~\ref{theorem1}
that if we restrict the kernel $K_0$ to $\gamma\times \gamma$ then the resulting integral operator satisfies better bounds than are posited in
\eqref{4.7}.  Specifically, we have
$$\Bigl(\int_{-\frac12}^{\frac12}\Bigl|\, \int_{-\frac12}^{\frac12}
K_0(\gamma(t),\gamma(s)) \, h(s)\, ds\, \Bigr|^2 \, dt\Bigr)^{\frac12} \le CT^{-1}\la \|h\|_{L^2([-\frac12,\frac12])}.$$
As a result, if we repeat arguments from the proof of Theorem~\ref{theorem2}, we find that if we set
\begin{equation}\label{4.8}
S_\la h(t)=\frac1{\pi T}\int_{-\infty}^\infty \hint (1-\beta(\tau)) \, \Hat \chi(\tau/T) e^{i\la\tau} \bigl(\cos \tau\sqd\bigr)(\gamma(t),\gamma(s))\, h(s) \, ds d\tau,
\end{equation}
then the following estimate would imply \eqref{4.7}
\begin{equation}\label{4.9}
\|S_\la h\|_{L^2([-\frac12,\frac12])}
\le C\la (T/\log(2+T))^{-1}\|h\|_{L^2([-\frac12, \frac12])} +C_T\la^{\frac34}\|h\|_{L^2([-\frac12,\frac12])}.
\end{equation}

By the Cartan-Hadamard theorem (see \cite{Chavel}, \cite{do Carmo}), the universal cover of $(M,g)$ is $(\Rth, \tilde g)$, where $\tilde g$ denotes
the pullback via the covering map $\kappa$  of the metric $g$ on $M$ to the universal cover $\Rth$.  Consequently, \eqref{3.7} is valid in the present context if $\Gamma$
now denotes the deck transformations for this universal cover.

Therefore, we can write
$$S_\la h(t)=\frac1{\pi T}\sum_{\alpha\in \Gamma}\int_{-\infty}^\infty \hint (1-\beta(\tau)) \, \Hat \chi(\tau/T) e^{i\la\tau} 
\bigl(\cos \tau\sqdt\bigr)(\tilde \gamma(t), \alpha(\tilde \gamma(s)))\, h(s) \, ds d\tau,$$
and, also, as before, write
$$S_\la h(t)=\Sstab h(t) + \Sosc h(t),\quad |t|\le 1/2,$$
with $\Sstab$ and $\Sosc$ defined by \eqref{3.9} and \eqref{3.10}, respectively.

If we use \eqref{2.14} in place of \eqref{ftcirc} we can repeat the proof of the 2-dimensional estimate \eqref{kest} to see that we now have
that 
the kernel $K^{\text{Stab}}_\la$ of $\Sstab$ satisfies
\begin{equation*}
|K^{\text{Stab}}_\la(t,s)|\le CT^{-1}\la \sum_{0\le j\le T+1}(1+j)^{-1}+C_T\le C\la(T/\log(2+T))^{-1}+C_T.
\end{equation*}
Consequently, 
\begin{equation*}
\|\Sstab h\|_{L^2([-\frac12,\frac12])}\le \bigl(C\la (T/\log(2+T))^{-1}+C_T\bigr)\, \|h\|_{L^{2}([-\frac12,\frac12])},
\end{equation*}
which is better than the bounds posited in \eqref{4.9} for the whole operator.

Based on this, we would have \eqref{4.9} if we could show that for every $\delta>0$
\begin{equation}\label{4.10}
\|\Sosc h\|_{L^2([-\frac12,\frac12])}\le \bigl(C_{T,\delta}\la^{\frac34}+\delta \la \bigr)\|h\|_{L^2([-\frac12,\frac12])}.
\end{equation}
Indeed, choosing $\delta=\log(2+T)/T$ here yields \eqref{4.9}.

To prove this we shall use the fact that if, once again, we repeat the proof of Lemma~\ref{lemma3.2} using \eqref{2.14} in place of \eqref{ftcirc}
we find that the kernel of each of the summands in the definition of $\Sosc$, which are given by the formula \eqref{3.11}, now satisfy
for $|s|, |t|\le 1/2$
\begin{equation}\label{4.11}
K^\gamma_{T,\la,\alpha}(t,s)=
%\\
w(\tilde \gamma(t),\alpha(\tilde \gamma(s)))
\sum_\pm a_\pm(T,\la; \,  \phi_{\gamma,\alpha}(t,s))e^{\pm i\la \phi_{\gamma,\alpha}(t,s)}
+R^\gamma_{T,\la,\alpha}(t,s),
\end{equation}
where, as before,
\begin{equation}\label{4.12}
\phi_{\gamma,\alpha}(t,s)=d_{\tilde g}(\tilde \gamma(t),\alpha(\tilde \gamma(s))), \quad |s|, |t|\le 1/2,
\end{equation}
but now \eqref{3.15} is replaced by
\begin{equation}\label{4.13}
|\partial_r^ja_\pm(T,\la; \, r)|\le C_jT^{-1}\la r^{-1-j}, \quad r\ge 1,
\end{equation}
while \eqref{3.16} and \eqref{3.18} remain valid.

It is at this point that we need to use the geometric lemma, Lemma~\ref{lemma4.1}, that relies on our assumption of constant curvature.  By
\eqref{4.1}, given $Id\ne \alpha \in \Gamma$ there is at most one exceptional value of $t=t(\gamma,\alpha)$ and one exceptional value
of $s=s(\gamma,\alpha)$ so that
\begin{multline}\label{4.14}
|\partial_t\partial_s\phi_{\gamma,\alpha}(t,s)|+|\partial^2_t\partial_s\phi_{\gamma,\alpha}(t,s)|+|\partial_t\partial^2_s\phi_{\gamma,\alpha}(t,s)|\ne 0,
\\
\text{if } \, \, (t,s)\in [-\frac12,\frac12]\times [-\frac12,\frac12], \quad \text{and } \, \, (t,s)\ne (t(\gamma,\alpha),s(\gamma,\alpha)).
\end{multline}

To exploit this, fix $b\in C^\infty({\mathbb R}^2)$ satisfying
\begin{equation}\label{4.15}
0\le b\le 1, \quad b(y)=0, \,\,  |y|\le \frac12, \quad \text{and } \, b(y)=1, \, \, |y|\ge 1.
\end{equation}
We then, given $\e>0$, split
\begin{equation}\label{4.16}
K^\gamma_{T,\la,\alpha}(t,s)=K^{\gamma,\e}_{T,\la,\alpha}(t,s)+R^{\gamma,\e}_{T,\la,\alpha}(t,s),
\end{equation}
where
$$K^{\gamma,\e}_{T,\la,\alpha}(t,s)=b\bigl(\e^{-1}((t,s)-(t(\gamma,\alpha),s(\gamma,\alpha))\bigr)\, K^\gamma_{T,\la,\alpha}(t,s),
$$
if there is such an exceptional $(t(\gamma,\alpha),s(\gamma,\alpha))$ and $K^{\gamma,\e}_{T,\la,\alpha}=K^\gamma_{T,\la,\alpha}$ if
there is none.

We then have, by \eqref{4.14}, that
$$|\partial_t\partial_s\phi_{\gamma,\alpha}|+|\partial^2_t\partial_s\phi_{\gamma,\alpha}|+|\partial_t\partial^2_s\phi_{\gamma,\alpha}|
\ne 0 \quad \text{on } \, \, \text{supp }K^{\gamma,\e}_{T,\la,\alpha}.$$
This means that the oscillatory integral operator $S^{\gamma,\e}_{T,\la,\alpha}$ with kernel $K^\gamma_{T,\la,\alpha}$ has a canonical 
relation for which either the right projection, $\Pi_R$, or the left projection, $\Pi_L$, is a submersion with at most fold singularities.  Therefore,
by \eqref{3.41} or \eqref{fold} we have that
$$\|S^{\gamma,\e}_{T,\la,\alpha}h\|_{L^2([-\frac12,\frac12])}\le C_{\alpha,\e}\la^{\frac34}\|h\|_{L^2([-\frac12,\frac12])}.$$
Since, as noted before, there are $O(\exp(cT))$ of these terms which are nonzero, we conclude that if
$$\tilde S^{\text{Osc}}_\la =\sum_{\alpha\in \Gamma}S^{\gamma,\e}_{T,\la,\alpha},$$
then we have
\begin{equation}\label{4.17}
\|\tilde S^{\text{Osc}}h\|_{L^2([-\frac12,\frac12])}\le C_{T,\e}\la^{\frac34}\|h\|_{L^2([-\frac12,\frac12])}.
\end{equation}

Clearly, by \eqref{4.13}, \eqref{4.15} and \eqref{4.16} and Young's inequality, each integral operator $R^{\gamma,\e}_{T,\la,\alpha}$ with kernel
$R^{\gamma,\e}_{T,\la,\alpha}(t,s)$ satisfies
$$\|R^{\gamma,\e}_{T,\la,\alpha}h\|_{L^2([-\frac12,\frac12])}\le C_\alpha \e \la \|h\|_{L^2([-\frac12,\frac12])},$$
with $C_\alpha$ independent of $\e$.  Thus, if
$$R^{\text{Osc}}_\la =\sum_{\alpha\in \Gamma}R^{\gamma,\e}_{T,\la,\alpha},$$
we have
\begin{equation}\label{4.18}
\|R^{\text{Osc}}_\la h\|_{L^2([-\frac12,\frac12])}\le C_T\e \la \|h\|_{L^2([-\frac12,\frac12])},
\end{equation}
where $C_T$ depends on $T$, but is independent of $\e$.

Since $\Sosc=\tilde S^{\text{Osc}}_\la +R^{\text{Osc}}_\la$, we obtain \eqref{4.10} from \eqref{4.17} and \eqref{4.18}.  Indeed, since $T$ is fixed,
we just need to choose $\e>0$ small enough so that $C_T\e<\delta$.  This completes the proof of \eqref{4.9} and, consequently, that of
\eqref{4.5}.

We have therefore shown that, under the assumption of constant nonpositive curvature, we have improved $L^2$-restriction estimates for
each fixed $\gamma\in \Pi$.  Since it is clear that the compactness arguments at the end of the proof of Theorem~\ref{theorem2} yield bounds
which are uniform as $\gamma$ ranges over $\Pi$, we conclude that we also have \eqref{1.6}, which completes the proof of Theorem~\ref{theorem3}.
\end{proof}

%\newtheorem*{rem}{Remark}

%%\begin{rem}
\bigskip
\noindent{\bf Remark.}
%We could have just used H\"ormander's oscillatory integral estimate \eqref{3.41} and not \eqref{fold} 
%and the fact that, by Lemma~\ref{lemma3.2}, the set of $(t,s)$ so that $\partial_t\partial_s\phi_{\gamma,\alpha}(t,s)\ne 0$ has measure zero to show
%that
%$$\|\Sosc h\|_{L^4([-\frac12,\frac12])}\le \e(\la)\la^{\frac14}\|h\|_{L^{\frac43}([-\frac12,\frac12])},$$
%with $\e(\la)\to 0$ as $\la \to +\infty$, which along with the easy estimates for the other piece, $\Sstab$, is enough to prove Theorem~\ref{theorem2}.
If we could prove that for metrics $\tilde g$ on ${\mathbb R}^3$ of nonpositive curvature we have that $\phi''_{ts}$ only vanishes on
a set of measure zero if $\phi$ 
associated to distinct geodesics on $({\mathbb R}^3,\tilde g)$
is defined as in Lemma~\ref{lemma3.2}, then it is clear that the proof of Theorem~\ref{theorem3} would show that its conclusions would be
valid in this case.  
Without making an assumption such as constant curvature,
this appears to be difficult to obtain.  
%Among other things,
%in 3-dimensions, as opposed to two dimensions, there are situations where $\phi''_{ts}$, $\phi'''_{tts}$ and $\phi''_{tss}$ can simultaneously vanish.  For instance, for the Euclidean metric this happens for the lines parameterized by $(t,0,0)$ and $(0,1,s)$ when $t=s=0$.  On the other hand, in the Euclidean 
%case, it is not too difficult to check directly that $\phi''_{ts}$ only vanishes on a set of measure zero if $\phi$ is the distance between points on two distinct lines measured by arclength.  
The fact that it seems difficult to prove an appropriate version of Lemma~\ref{lemma3.2} in three-dimensions 
or of Lemma~\ref{lemma4.1} without the constant curvature assumptions
is also related to the fact that
what is referred to as the $n\times n$ Carleson-Sj\"olin condition in \cite{mss} and \cite{S3} is much more complicated in three or more dimensions
as opposed to two-dimensions.  Counterexamples for certain types of 
natural oscillatory integral bounds of  Minicozzi and the second author~\cite{MinS} (which were obtained independently later
also by Bourgain and Guth~\cite{BGu} as well as much harder positive results) are a manifestation of this; however, the negative results in 
\cite{BGu} and \cite{MinS} are probably not related to whether or not a suitable version of Lemma~\ref{lemma3.2} is valid in three-dimensions.
As was shown in \cite{MinS} and  \cite{SoJAMS} one tends to avoid these pathologies when there are many totally geodesic submanifolds.  
In the case of odd dimensions $n$ the results in \cite{MinS} show that certain oscillatory integral estimates involving the Riemannian distance function as phase can break down when there are not many totally geodesic submanifolds of dimension $(n+1)/2$, while in even dimensions this may happen when there are not many totally geodesic submanifolds of dimension $(n+2)/2$.
%Based on this it is natural to question whether Theorem~\ref{theorem3} is also valid when $(M,g)$ is a compact 3-dimensional locally symmetric Riemannian manifold since then
%its universal cover $({\mathbb R}^3,\tilde g)$ possesses many such submanifolds---see e.g., \cite{Helgason}.
%%\end{rem}

\end{document}